\documentclass[12pt]{amsart}

\usepackage{graphicx}
\usepackage{amssymb}
\usepackage{float}
\usepackage{fullpage}

\newtheorem{theorem}{Theorem}[section]
\newtheorem{lemma}[theorem]{Lemma}
\newtheorem{corollary}[theorem]{Corollary}
\newtheorem{proposition}[theorem]{Proposition}

\theoremstyle{definition}
\newtheorem{definition}[theorem]{Definition}

\theoremstyle{remark}
\newtheorem{remark}[theorem]{Remark}

\numberwithin{equation}{section}

\usepackage{physics}

\usepackage{comment}

\usepackage{mathtools}

\newcommand{\C}{{\mathbb C}}
\newcommand{\R}{{\mathbb R}}
\newcommand{\Q}{{\mathbb Q}}
\newcommand{\Z}{{\mathbb Z}}
\newcommand{\N}{{\mathbb N}}

\newcommand{\D}{{\mathbb D}}

\renewcommand{\phi}{{\varphi}}
\renewcommand{\le}{{\,\leqslant}}
\renewcommand{\ge}{{\,\geqslant}}

\newcommand{\1}{{\mathbf 1}}

\newcommand{\eps}{{\varepsilon}}

\begin{document}

\title[M\"obius on shifted primes]{Averages of the M\"obius function on shifted primes}

\author{Jared Duker Lichtman}
\address{Mathematical Institute, University of Oxford, Oxford, OX2 6GG, UK}

\email{jared.d.lichtman@gmail.com}



\date{October 20, 2021.}



\begin{abstract}
It is a folklore conjecture that the M\"obius function exhibits cancellation on shifted primes; that is, $\sum_{p\le X}\mu(p+h) \ = \ o(\pi(X))$ as $X\to\infty$ for any fixed shift $h>0$. This appears in print at least since Hildebrand in 1989. We prove the conjecture on average for shifts $h\le H$, provided $\log H/\log\log X\to\infty$. We also obtain results for shifts of prime $k$-tuples, and for higher correlations of M\"obius with von Mangoldt and divisor functions. Our argument combines sieve methods with a refinement of Matom\"aki, Radziwi\l\l, and Tao's work on an averaged form of Chowla's conjecture.
\end{abstract}

\maketitle


\section{Introduction}

Let $\mu:\N\to\{-1,0,+1\}$ denote the M\"obius function, defined multiplicatively on primes $p$ by $\mu(p)=-1$ and $\mu(p^k)=0$ for $k\ge2$. Many central results in number theory may be formulated in terms of averages of the M\"obius function. Notably, the prime number theorem is equivalent to the statement $\sum_{n\le X}\mu(n) = o(X)$, and $\sum_{n\le X}\mu(n) = O(X^\theta)$ for all $\theta>\frac{1}{2}$ is equivalent to the Riemann hypothesis.

Clearly $\mu(p)=-1$ gives $\sum_{p\le X}\mu(p) = -\pi(X)$, but less is known about the M\"obius function on shifted primes. It is a folklore conjecture that $\sum_{p\le X}\mu(p+h) \ = \ o(\pi(X))$ for any fixed shift $h>0$. This appeared in print at least since Hildebrand \cite[p.212]{Hild}, as well as Sarnak \cite[Problem 5.2]{Sarnk} and in Murty--Vatwani \cite[(1.2)]{MurtyV}). We answer an averaged version of this conjecture with quantitative bounds.

\begin{theorem}\label{thm:main}
If $H<X$ and $\log H/\log_2 X \to\infty$ as $X\to\infty$, then
\begin{align}\label{eq:mainqual}
\sum_{h\le H}\bigg|\sum_{p\le X}\mu(p+h)\bigg| = o(H\pi(X)).
\end{align}
Further if $H = X^\theta$ for some $\theta\in (0,1)$, then for all $\delta>0$
\begin{align*}
\sum_{h\le H}\bigg|\sum_{p\le X}\mu(p+h)\bigg| \ \ll_{\theta,\delta} \ \frac{H\pi(X)}{(\log X)^{1/3-\delta}}.
\end{align*}
\end{theorem}

An immediate consequence is that
$\sum_{p<X}\mu(p+h)$ exhibits cancellation for all but $o(H)$ values of $h\le
H=(\log X)^{\psi(X)}$ provided $\psi(X)\rightarrow \infty$.

\begin{remark}
The weaker qualitative cancellation \eqref{eq:mainqual} in the longer regime $H=X^\theta$ may be obtained more directly, using a recent Fourier uniformity result of Matom\"aki--Radziwi\l\l--Tao \cite{MRTUnif}. See Theorem \ref{thm:pretend} for details.
\end{remark}

Theorem \ref{thm:main} is an illustrative example within a broader class of correlations that may be handled by the methods in this paper, see Theorem \ref{thm:mainChowla} for the full technical result. Below we highlight some further example correlations of general interest.

\subsection{Higher correlations}

The influential conjectures of Chowla \cite{Chowla} and Hardy--Littlewood \cite{HardLittl} assert that for any fixed tuple $\mathcal H = \{h_1,.., h_k\}$ of distinct integers,
\begin{align*}
\sum_{n\le X}\mu(n+h_1)\cdots\mu(n+h_k) \ &= \ o(X),\\
\sum_{n\le X}\Lambda(n+h_1)\cdots\Lambda(n+h_k) \ & = \ \frak{S}(\mathcal H)X \ + \ o(X),
\end{align*}
for the singular series $\frak{S}(\mathcal H) = \prod_{p}\frac{(1-\nu_p/p)}{(1-1/p)^k}$, where $\nu_p = \#\{h_1,..,h_k (\text{mod }p)\}$. Both conjectures remain open for any $k\ge2$.

We establish an average result for Hardy--Littlewood--Chowla correlations.
\begin{theorem}\label{thm:Chowktuple}
Suppose $(\log X)^{300}<H<X$, and write $H=(\log X)^{\psi(X)}$.
Then for any $\delta>0$, $m,k\ge1$, and fixed tuple $\mathcal A = \{a_1,..,a_k\}$ of disinct integers, we have
\begin{align}\label{thm:Chowktuple}
\sum_{h_1,..,h_m\le H}\bigg|\sum_{n\le X} \prod_{j=1}^m \mu(n+h_j)\prod_{i=1}^k\Lambda(n+a_i)\bigg| \ \ll_{\delta,m,\mathcal A} \
\frac{XH^m}{\min\{\psi(X)^m, (\log X)^{m/3-\delta}\}}.
\end{align}
\end{theorem}

It is worth emphasizing particular aspects of this result. First, \eqref{thm:Chowktuple} holds for an arbitrary fixed prime $k$-tuple. We must average over at least $m\ge1$ copies of $\mu$ in order to obtain cancellation. Notably, the cancellation becomes quantitatively stronger for larger $m$, e.g. error savings $(\log X)^{m/3-\delta}$. For the case $m=0, k=2$, Matom\"aki--Radziwi\l\l--Tao \cite{MRTCor} handled binary correlations $\sum_{n\le X}\Lambda(n)\Lambda(n+h)$ on average with error savings $(\log X)^A$, though in the much larger regime $H \ge X^{8/33+\eps}$.

In particular, the Chowla conjecture holds on average along the subsequence of primes.
\begin{corollary}\label{cor:Chow}
Suppose $H<X$ and $\log H/\log_2 X\to\infty$ as $X\to\infty$. Then for any $m\ge1$,
\begin{align}
\sum_{h_1,..,h_m\le H}\bigg|\sum_{p\le X}\mu(p+h_1)\cdots\mu(p+h_m)\bigg| \ = \ o_m\big(\pi(X)H^m\big).
\end{align}
\end{corollary}

Moreover, using Markov's inequality we may obtain qualitative cancellation for almost all shifts, with arbitrary log factor savings in the exceptional set. 

\begin{corollary}\label{cor:exceptmu}
Suppose $H<X$ and $\log H/\log_2 X\to\infty$ as $X\to\infty$. Then for any $A>0$,
\begin{align*}
\sum_{p\le X}\mu(p+h_1)\cdots\mu(p+h_m) \ = \ o_m(\pi(X)),
\end{align*}
for all except $O_A(H^m(\log X)^{-A})$ shifts $(h_1,..,h_m)\in [1,H]^m$.
\end{corollary}

These results build on earlier work of Matom\"aki--Radziwi\l\l--Tao \cite{MRTChow}, who established an average form of Chowla's conjecture,
\begin{align}\label{eq:MRTchow}
\sum_{h_1,..,h_m\le H}\bigg|\sum_{n\le X}\mu(n+h_1)\cdots\mu(n+h_m)\bigg| \ = \ o_m(XH^m),
\end{align}
for any $H=H(X)\to\infty$ arbitrarily slowly. Whereas, our results require the faster growth $H=(\log X)^{\psi(X)}$ with $\psi(X)\to\infty$ arbitrarily slowly.

\subsection{Correlations with divisor functions}
Consider fixed integers $a\ge1$, $k \ge l\ge 2$. The well studied correlation of two divisor functions $d_k,d_l$ is predicted to satisfy
\begin{align*}
\sum_{n\le X}d_k(n+h)d_l(n) \ = \ C_{k,l,h} \cdot \big(X \ + \ o(X)\big)(\log X)^{k-l-2},
\end{align*}
for a certain (explicit) constant $C_{k,l,h}>0$. Recently, Matom\"aki--Radziwi\l\l--Tao \cite{MRTCor2} have shown the following averaged result, in the regime $H \ge (\log X)^{10000 k\log k}$, 
\begin{align*}
\sum_{h\le H}\Big|\sum_{n\le X}d_k(n+h)d_l(n) \ - \ C_{k,l,h}\cdot X(\log X)^{k-l-2}\Big| \ &= \ o_k(HX(\log X)^{k+l-2}).
\end{align*}

For higher correlations of divisor functions with M\"obius, we obtain the following.
\begin{theorem}\label{cor:divisorcorr}
For any $j\ge1$, $k_1,..,k_j\ge2$, let $k=\sum_{i=1}^j k_i$ and take any fixed tuple $\mathcal A = \{a_1,..,a_j\}$ of distinct integers. If $H<X$ and $\log H/\log_2 X \to\infty$, then
\begin{align*}
\sum_{h\le H}\Big|\sum_{n\le X}\mu(n+h)\prod_{i=1}^jd_{k_i}(n+a_i)\Big| \ & = \ o_{k,\mathcal A}\big(HX(\log X)^{k-j}\big).
\end{align*}
\end{theorem}
Again, we emphasize the need to average over the shift $h$ that inputs to M\"obius $\mu(n+h)$, while $a_i$ may be fixed arbitrarily.

\begin{remark}
For simplicity, the results are stated for the M\"obius function $\mu$, but our results hold equally for its completely multiplicative counterpart, the Liouville function $\lambda$. In fact, the proof strategy is to reduce from $\mu$ to $\lambda$.
\end{remark}

The main number-theoretic input is the classical Vinogradov--Korobov zero-free region
\begin{align}\label{eq:VinKor}
\bigg\{\sigma + it \; : \; 1-\sigma < \frac{c}{\max\big\{\log q,\,\log(|t|+3)^{2/3}\log\log(|t|+3)^{1/3}\big\}} \bigg\}
\end{align}
for $L(s,\chi)$, where $\chi$ is a Dirichlet character of modulus $q\le (\log X)^A$ in the Siegel--Walfisz range, see \cite[\S8]{IK}.

\subsection{Beyond M\"obius}

We also consider general multiplicative functions $f:\N\to\C$, which do not pretend to be a character $f(n)\approx n^{it}\chi(n)$ for some $\chi$ (mod $q$). More precisely, we follow Granville and Soundararajan \cite{GranSound} and define the pretentious distance
\begin{align*}
\D(f,g;X) = \bigg(\sum_{p\le X}\frac{1-\Re(f(p)\overline{g(p)})}{p}\bigg)^{1/2},
\end{align*}
and the related quantity
\begin{align}
M(f;X,Q) = \inf_{\substack{|t|\le X\\\chi\,(q), \, q\le Q}} \D\big(f,\, n\mapsto n^{it}\chi(n); X\big)^2.
\end{align}

We may apply recent work of Matom\"aki--Radziwi\l\l--Tao \cite{MRTUnif} on Fourier uniformity, in order to more directly obtain (qualitative) cancellation for averages of non-pretentious multiplicative functions over shifted primes.

\begin{theorem}\label{thm:pretend}
Given $\theta\in(0,1)$ let $H = X^\theta$. Given a multiplicative function $f:\N\to\C$ with $|f|\le1$. There exists $\rho\in (0,\frac{1}{8})$ such that, if $M(f;X^2/H^{2-\rho},Q) \to \infty$ as $X\to \infty$ for each fixed $Q>1$, then
\begin{align*}
\sum_{h\le H}\Big|\sum_{p\le X}f(p+h)\Big| \ = \ o_{\theta,\rho}\big(H\pi(X)\big).
\end{align*}
\end{theorem}

In particular, $f=\mu$ does not pretend to be a Dirichlet character, a fact equivalent to the prime number theorem in arithmetic progressions. Indeed,
\begin{align*}
M(\mu; X,Q) \ge \inf_{\substack{|t|\le X\\\chi\,(q), \, q\le Q}} \sum_{e^{(\log X)^{2/3+\eps}}\le p\le X}\frac{1+\Re\chi(p)p^{it}}{p} \ \ge \ \Big(\frac{1}{3}-\eps\Big)\log\log X + O(1),
\end{align*}
where the latter inequality is well-known to follow from the zero-free region \eqref{eq:VinKor}.

\subsection{Overview of the proof of Theorem \ref{thm:main}}

We now indicate the general form of the proof. We pursue a variation on the approach of Matom\"aki--Radziwi\l\l--Tao \cite{MRTChow}. Namely, we first restrict \eqref{eq:mainqual} to `typical' terms $\mu(n)$ for $n=p+h\in\mathcal S$ that have prime factors lying in certain prescribed intervals $[P_1,Q_1], [P_2,Q_2]$. The terms with $n\notin \mathcal S$ are sparse, and thus may be shown to contribute negligibly by standard sieve estimates. (For higher correlations, one may also use sieve estimates, along with work of Henriot \cite{Hen} to handle a general class of functions with `moderate growth' that are `amenable to sieves.')

Once reduced to numbers with `typical factorization,' we decouple the short interval correlation between M\"obius and the indicator for the primes, using a Fourier identity and applying Cauchy--Schwarz (Lemma \ref{lem:Fourier}). This essentially yields a bound of $\pi(X)\ll X/\log X$ times a Fourier-type integral for $\mu$,
\begin{align*}
\sup_\alpha \int_0^X \bigg|\sum_{\substack{x\le n\le x+H\\n\in \mathcal S}} \mu(n)e(n\alpha)\bigg|\dd{x}.
\end{align*}

This decoupling step is a gambit. It has the advantage of only needing to consider $\mu$ on its own, but loses a factor of $\log X$ from the density of the primes. To make this gambit worthwhile, we must recover over a factor of $\log X$ savings in the above Fourier integral for $\mu$. However, Matom\"aki--Radziwi\l\l--Tao \cite[Theorem 2.3]{MRTChow} bound the above integral with roughly $(\log X)^{\frac{1}{500}}$ savings (though their bound holds for any non-pretentious multiplicative function $g$.) Therefore we must refine the argument in the special case of $g=\mu$ to win back over a full factor of $\log X$. We note this task is impossible unless $H$ is larger than a power of log.

We accomplish this task in the `key Fourier estimate' (Theorem \ref{thm:mainFourier}), which bounds the above integral with $(\log X)^A$ savings for any $A>0$ (though $\mathcal S$ will implicitly depend on $A$). As with \cite{MRTChow}, this bound is proven by reducing to the analogous estimate with the completely multiplicative Liouville function $\lambda$, and splitting up $\alpha\in[0,1]$ into major and minor arcs. 

The main technical innovation here comes from the major arcs (Proposition \ref{prop:MRmain}), essentially saving a factor $(\log X)^A$ in the mean values of `typical' Dirichlet polynomials of the form
\begin{align*}
\sum_{\substack{X\le n\le 2X\\n\in \mathcal S}}\frac{\lambda(n)\chi(n)}{n^s}
\end{align*}
for a character $\chi$ of modulus $q\le (\log X)^A$ in the Siegel--Walfisz range. This refines the seminal work of Matom\"aki--Radziwi\l\l \ \cite{MR}, who obtained a fractional power of log savings for the corresponding mean values. However, Matom\"aki--Radziwi\l\l's results apply to the general setting of (non-pretentious) multiplicative functions and appeal to Hal\'asz's theorem, which offers small savings. By contrast, our specialization to the M\"obius function affords us the full strength of Vinogradov--Korobov estimates (Lemma \ref{lem:MR2l}).

The Matom\"aki--Radziwi\l\l \ method saves roughly a fractional power of $P_1$ in the Dirichlet mean value when $Q_1\approx H$. So in order to recover from our initial gambit, we are prompted to choose $P_1=(\log X)^C$ for some large $C>0$. Then by a standard sieve bound the size of $\overline{\mathcal S}$ is morally $O(\frac{\log P_1}{\log Q_1}) = O_C(\frac{\log\log X}{\log H})$. This highlights the need for our assumption $\log H/\log\log X\to\infty$.

We remark that the Matom\"aki--Radziwi\l\l \ method requires two intervals $[P_1,Q_1], [P_2,Q_2]$ (that define $\mathcal S$) in order to handle `typical' Dirichlet polynomials in the regime $H=(\log X)^{\psi(X)}$ for $\psi(X)\to\infty$. Note in general \cite{MR} the slower $H\to\infty$ the more intervals we require (though by a neat short argument \cite{MRshort}, only one interval is needed in the regime $H=X^\theta$ for $\theta>0$).



\section*{Notation}
We recall the key arithmetic functions. The M\"obius function $\mu$ is defined multiplicatively from primes $p$ by $\mu(p)=-1$ and $\mu(p^k)=0$ for $k\ge2$. Similarly the Liouville function $\lambda$ is defined {\it completely} multiplicatively by $\lambda(p)=-1$. The von Mangoldt function is given by $\Lambda(n) = \log p$ on prime powers $n=p^k$, and zero otherwise. For $l\ge2$ the $l$th divisor function is $d_l(n)=\sum_{n=n_1\cdots n_l}1$

We use standard asymptotic notation: $X\ll Y$ and $X = O(Y)$ both mean $|X|\le CY$ for some some absolute constant $C$, and $X\asymp Y$ means $X\ll Y\ll X$. If $x$ is a parameter tending to infinity, $X = o(Y)$ means that $|X| \le c(x)Y$ for some quantity $c(x)$ that tends to zero as $x\to\infty$. Let $\log_k X = \log_{k-1} (\log X)$ denote the $k$th-iterated logarithm.

Unless otherwise specified, all sums range over the integers, except for sums over the variable $p$ (or $p_1$, $p_2$,..) which are understood to be over the set of primes $\mathbb P$. Let $e(x) := e^{2\pi i x}$.

We use $\1_S$ to denote the indicator of a predicate $S$, so $\1_S = 1$ if $S$ is true and $\1_S = 0$ if $S$ is false. When $\mathcal S$ is a set, we write $\1_{\mathcal S}(n) = \1_{n\in \mathcal S}$ as the indicator function of $\mathcal S$. Also let $\1_\mathcal S f$ denote the function $n\mapsto\1_\mathcal S(n) f(n)$.

\section{Initial reductions}

In this section, we shall make some initial reductions along the lines of Matom\"aki--Radziwi\l\l--Tao \cite[Theorem 2.3]{MRTChow}. We shall restrict our attention to numbers $n=p+h$ with prime factors in prescribed intervals $[P_1,Q_1],[P_1,Q_1]$ (as defined in \eqref{eq:PQj}). The exceptional $n$ are rare and thus contribute negligibly, as shown by a standard sieve bound. From here, we shall decouple the correlation between $\mu$ and the indicator for the primes, and reduce the problem to the key Fourier estimate for $\mu$, as in Theorem \ref{thm:mainFourier}. For technical convenience, we shall further reduce to $\lambda$.

To this, we begin with a general Fourier-type bound to decouple correlations of arbitrary functions.
\begin{lemma}[Fourier bound]\label{lem:Fourier}
Given $f,g:\N\to\C$, let $F(X) := \sum_{n\le X}|f(n)|^2$. Then
\begin{align}\label{eq:mainS}
\sum_{|h|\le H}\bigg|\sum_{n\le X}f(n)\,g(n+h)\bigg|^2 \ \ll \ F(X+2H)\cdot\sup_\alpha \int_0^X \bigg|\sum_{x\le n\le x+2H}g(n)e(n\alpha)\bigg|\dd{x}.
\end{align}
\end{lemma}
\begin{proof}
First, the lefthand side of \eqref{eq:mainS} is
\begin{align}
\sum_{|h|\le H}\bigg|\sum_{n\le X}f(n)\,g(n+h)\bigg|^2  \ll \ H^{-2}\sum_{|h|\le 2H}(2H-|h|)^2\bigg|\sum_{n\le X}f(n)\,g(n+h)\bigg|^2 \ =: H^{-2}\,\Sigma \label{eq:step1}.
\end{align}

Expanding the square in $\Sigma$ and letting $h = m-n=m'-n'$, we have
\begin{align*}
\Sigma &
= \sum_{|h|\le 2H}(2H-|h|)^2 \sum_{n,n'\le X} f(n)\overline{f}(n')g(n+h)\overline{g}(n'+h)\\ 
& \ = \sum_{n,n'\le X}\sum_{m,m'}f(n)\overline{f}(n')g(m)\overline{g}(m') \1_{m-n=m'-n'} \cdot \Big(\int_0^X \1_{x\le n,m\le x+2H}\dd x\Big)\Big(\int_0^X \1_{x'\le n',m'\le x'+2H}\dd x'\Big).
\end{align*}
Then orthogonality $\1_{m-n=m'-n'} = \int_0^1e((m-n-m'+n')\alpha)\dd{\alpha}$ gives
\begin{align*}
\Sigma = \int_0^1 & \bigg(\int_0^X \sum_{x\le n,m\le x+2H}f(n)g(m)e\big((m-n)\alpha\big)\dd{x}\bigg) \\
\cdot \ & \bigg(\int_0^X \sum_{x'\le n',m'\le x'+2H} \overline{f}(n')\overline{g}(m') e\big((n'-m')\alpha\big)\dd{x'}\bigg) \dd{\alpha} \nonumber\\
& = \int_0^1 \bigg|\int_0^X \sum_{x\le n,m\le x+2H}f(n)g(m)e\big((m-n)\alpha\big)\dd{x}\bigg|^2 \dd{\alpha}.
\end{align*}
Using Cauchy--Schwarz, we bound $\Sigma$ as
\begin{align}
\Sigma \ & \le \ \int_0^1 \int_0^X \bigg|\sum_{x\le m\le x+2H}g(m)e(m\alpha)\bigg|^2\dd{x}\cdot\int_0^X \bigg|\sum_{y\le n\le y+2H}\overline{f}(n)e(n\alpha)\bigg|^2\dd{y} \dd{\alpha} \nonumber\\
& \ \ll \ H\bigg(\sup_\alpha \int_0^X \bigg|\sum_{x\le m\le x+2H} g(m)e(m\alpha)\bigg|\dd{x}\bigg) \int_0^1\int_0^X \bigg|\sum_{y\le n\le y+2H}\overline{f}(n)e(n\alpha)\bigg|^2\dd{y}\dd{\alpha}. \label{eq:step3}
\end{align}
Using $\int_0^1 e(n\alpha)\dd{\alpha} = \1_{n=0}$ again, the second integral in \eqref{eq:step3} is
\begin{align*}
\int_0^1\int_0^X\bigg|\sum_{y\le n\le y+2H} & \overline{f}(n)e(n\alpha)\bigg|^2  \dd{y}\dd{\alpha}
= \int_0^X\sum_{y\le n,n'\le y+2H} \overline{f}(n)f(n')\int_0^1 e\big((n-n')\alpha\big)\dd{\alpha}\dd{y}\\
& = \int_0^X\sum_{y\le n\le y+2H}|f(n)|^2\dd{y} 
= \sum_{n\le X+2H}|f(n)|^2\int_{n-2H}^n\dd{y}
\ll HF(X+2H).
\end{align*}
Hence plugging the bound \eqref{eq:step3} for $H^{-2}\Sigma$ back into \eqref{eq:step1} gives the result.
\end{proof}

Next we consider numbers with `typical factorization.'

For $A,\delta\ge 0$, define $\psi$ via $H = (\log X)^{\psi(X)}$ and consider the intervals
\begin{align}\label{eq:PQj}
[P_1,Q_1] &= [(\log X)^{33A},(\log X)^{\psi(X)-4A}],\\
[P_2,Q_2] &= [\exp\big((\log X)^{2/3+\delta/2}\big),\,\exp\big((\log X)^{1-\delta/2}\big)], \nonumber
\end{align}
and define the `typical factorization' set
\begin{align}\label{eq:S}
\mathcal S \ = \ \mathcal S(X,H,A,\delta) := \{n\le X : \exists \text{ prime factors } p_1,p_2\mid n \text{ with }p_j\in[P_j,Q_j]\}.
\end{align}

Using the Fourier bound, we shall reduce Theorem \ref{thm:main} to the following.
\begin{theorem}[Key Fourier estimate for $\mu$] \label{thm:mainFourier}
Given any $A>5$, $\delta>0$, let $\mathcal S \ = \ \mathcal S(X,H,A,\delta)$ as in \eqref{eq:S}. If $(\log X)^{40A} < H < X$, then
\begin{align*}
\sup_\alpha \int_0^X \bigg|\sum_{\substack{x\le n\le x+H\\n\in \mathcal S}} \mu(n)e(n\alpha)\bigg|\dd{x} \ \ll_{A,\delta} \ \frac{HX}{(\log X)^{A/5}}.
\end{align*}
\end{theorem}
\begin{proof}[Proof of Theorem \ref{thm:main} from Theorem \ref{thm:mainFourier}]
By a standard sieve upper bound \cite[Theorem 7.4]{Opera}, for each $h\,\le H$, $j=1,2$ we have
\begin{align}\label{eq:sievePQ}
\#\{\,p\,\le X \; : \; q\nmid p+h\,\forall q\in [P_j,Q_j]\} \ \ll \ \pi(X)\frac{\log P_j}{\log Q_j}\frac{h}{\phi(h)}.
\end{align}
Thus, recalling the choice of $[P_j,Q_j]$ in \eqref{eq:PQj}, the terms $p+h\notin\mathcal S$ trivially contribute to \eqref{eq:mainqual}
\begin{align}\label{eq:sieveph}
\sum_{h\le H} \Big|\sum_{\substack{p\le X\\p+h\notin\mathcal S}} \mu(p+h)\Big| \ \le  \sum_{1\le j\le2}\underset{q\nmid p+h\,\forall q\in [P_j,Q_j]}{\sum_{h\le H}\sum_{p\le X}}1
& \ll_{\delta} \ \pi(X)\Big(\frac{A}{\psi(X)} + (\log X)^{\delta-1/3}\Big)\sum_{h\le H}\frac{h}{\phi(h)} \nonumber\\
& \ll_{A,\delta} \ \frac{H\pi(X)}{\min\{\psi(X), (\log X)^{1/3-\delta}\}}.
\end{align}

On the other hand for $p+h\in \mathcal S$, Lemma \ref{lem:Fourier} with $f(n)=\1_{\mathbb P}(n)$, $g(n) = \1_{\mathcal S}\mu(n)$ gives
\begin{align*}
\sum_{h\le H}\bigg|\sum_{\substack{p\le X\\p+h\in\mathcal S}}\mu(p+h)\bigg|^2 \ll \pi(X+2H)\cdot\sup_\alpha \int_0^X \bigg|\sum_{\substack{x\le n\le x+2H\\n\in\mathcal S}} \mu(n)e(n\alpha)\bigg|\dd{x} \ \ll_{A,\delta} \ \frac{HX^2}{(\log X)^{A/5+1}},
\end{align*}
assuming Theorem \ref{thm:mainFourier}. Thus by Cauchy--Schwarz we obtain
\begin{align}\label{eq:inS}
\sum_{h\le H}\bigg|\sum_{\substack{p\le X\\p+h\in\mathcal S}}\mu(p+h)\bigg| 
\ & \ll \ \bigg(H\sum_{h\le H}\bigg|\sum_{\substack{p\le X\\p+h\in\mathcal S}}\mu(p+h)\bigg|^2\;\bigg)^{1/2}\ \ll_{A,\delta} \ \frac{HX}{(\log X)^{A/10+1/2}}.
\end{align} 
Hence \eqref{eq:sieveph} and \eqref{eq:inS} with $A = 6$ give Theorem \ref{thm:main}.
\end{proof}

Let $W = (\log X)^A$. Recall Theorem \ref{thm:mainFourier} asserts that $M_H(X;X) \ll_{A,\delta} XH/W^{1/5}$ for
\begin{align*}
M_{H}(X;Y) := \sup_\alpha \int_0^X \bigg|\sum_{\substack{x\le n\le x+H\\n\in \mathcal S(Y,A,\delta)}} \mu(n)e(n\alpha)\bigg|\dd{x}.
\end{align*}
We first note, that, for technical convenience, it suffices to establish $M_{H_0}(2X;X) \ll XH_0/W^{1/5}$ with $H_0:=\min\{H,\exp((\log X)^{2/3})\}$. Indeed, if $H\in [H_0,X]$ then by the triangle inequality
\begin{align*}
M_{H}(X;X) \le \sum_{k\le \lceil H/H_0\rceil}M_{H_0}(X+kH_0;X) \ll \sum_{k\le \lceil H/H_0\rceil}\frac{(X+kH_0)H_0}{W^{1/5}} \ \ll \ \frac{H}{H_0}\cdot \frac{XH_0}{W^{1/5}} = \frac{XH}{W^{1/5}}
\end{align*}
as desired. Hence we may assume $H\le \exp((\log X)^{2/3})$ hereafter. This reduction is not strictly necessary, but will simplify the argument. For example, in this case $Q_1<P_2$ so the intervals $[P_j,Q_j]$ are disjoint.

Consider the `refined typical factorization' sets $\mathcal S_d=\{n/d : d\mid n\in \mathcal S\}$ for $d< P_1$, that is,
\begin{align}\label{eq:Sd}
\mathcal S_d \ = \ \mathcal S_d(X,H,A,\delta) = \{m\le X/d : \exists\text{ prime factors } p_1,p_2\mid m \text{ with }p_j\in[P_j,Q_j]\}.
\end{align}

So far we have reduced Theorem \ref{thm:main} to Theorem \ref{thm:mainFourier} for $\mu$. We now reduce further to the analogous estimate for its completely multiplicative counterpart $\lambda$.

\begin{proposition}[Key Fourier estimate for $\lambda$] \label{prop:compmult}
Given any $A>5$, $\delta>0$, $H=(\log X)^{\psi(X)}$ with $40A \le \psi(X)\le (\log X)^{2/3}$. For $d\le W=(\log X)^A$ and $\mathcal S_d = \mathcal S_d(X,H,A,\delta)$ as in \eqref{eq:Sd}, we have
\begin{align*}
\sup_{\alpha}\int_0^{2X} \bigg|\sum_{\substack{x\le nd\le x+H\\n\in \mathcal S_d}} \lambda(n)e(\alpha n)\bigg|\dd{x} \ \ll_{A,\delta} \ \frac{HX}{d^{3/4}W^{1/5}}.
\end{align*}
\end{proposition}
\begin{proof}[Proof of Theorem \ref{thm:mainFourier} from Proposition \ref{prop:compmult}.]
By M\"obius inversion, we have $\mu=\lambda\ast h$ for $h = \mu\ast (\mu\lambda)$, where $\ast$ denotes Dirichlet convolution. That is, $h(d^2) = \mu(d)$ for squarefree $d$, and zero otherwise. Thus we may write
\begin{align*}
\sum_{x\le n\le x+H}\1_{\mathcal S}(n)\mu(n)e(n\alpha)
& = \sum_{d\ge 1} h(d) \sum_{x\le md\le x+H}\1_{\mathcal S}(md)\lambda(m)e(md\alpha),
\end{align*}
and so the triangle inequality gives
\begin{align}
\int_0^{2X} \bigg|\sum_{x\le n\le x+H}\1_{\mathcal S}(n)\mu(n)e(n\alpha)\bigg|\dd{x} & \le \sum_{d\ge 1} |h(d)| \int_0^{2X} \bigg|\sum_{x\le md\le x+H}\1_{\mathcal S}(md)\lambda(m)e(md\alpha)\bigg|\dd{x}. \label{eq:intSfe}
\end{align}
Note, using the trivial bound and swapping the order of summation and integration, the contribution of $d> W$ to \eqref{eq:intSfe} is 
\begin{align}\label{eq:moreW}
\ll \ \sum_{W<d\le 2X} |h(d)| \sum_{md\le X+H}H \ \ll \ \frac{HX}{W^{1/4}}\sum_{d\ge 1} \frac{|h(d)|}{d^{3/4}} \ \ll \ \frac{HX}{W^{1/4}}\sum_{d\ge 1} d^{-3/2} \ \ll \ \frac{HX}{W^{1/4}},
\end{align}
since $|h|\le1$ is supported on squares.

On the other hand the contribution of $d\le W$ to \eqref{eq:intSfe} is
\begin{align}\label{eq:lessW}
\le \sum_{d\ge 1} |h(d)| \int_0^{2X} \bigg|\sum_{x\le md\le x+H}\1_{\mathcal S_d}(m)\lambda(m)e(md\alpha)\bigg|\dd{x}
 \ll \frac{HX}{W^{1/5}}\sum_{d\le W} \frac{|h(d)|}{d^{3/4}} \ \ll \ \frac{HX}{W^{1/5}}
\end{align}
assuming Proposition \ref{prop:compmult}, and noting $\1_{\mathcal S}(md) = \1_{\mathcal S_d}(m)$ since $d\le W<P_1$. Together \eqref{eq:moreW} and \eqref{eq:lessW} give Theorem \ref{thm:mainFourier}.
\end{proof}

\begin{remark}
In \cite[Theorem 2.3]{MRTChow}, Matom\"aki--Radziwi\l\l--Tao bound the key Fourier integral for $\lambda$ with roughly $(\log X)^{\frac{1}{500}}$ savings over the trivial bound $HX$. Indeed their bound holds in general for any 1-bounded, non-pretentious multiplicative function $g$. However, the main point here is that the specific choice of $g=\lambda$ (and by extension, $\mu$) will allow us to extract a savings of any higher power of $\log x$.
\end{remark}


\section{Key Fourier estimate}

In this section, we establish Proposition \ref{prop:compmult} by the circle method, following the argument in \cite[Proposition 2.4]{MRTChow}.

Take $\alpha\in [0,1]$. By Dirichlet's approximation theorem there exists $\frac{a}{q}\in\Q$ with $(a,q)=1$ and $1\le q\le Q_1$ for which
\begin{align*}
\Big|\alpha - \frac{a}{q}\Big| \ \le  \ \frac{1}{qQ_1}.
\end{align*}
So we may split $[0,1]$ into major arcs $\frak M$ and minor arcs $\frak m$, according to the size of denominator $q$ compared to $W$,
\begin{align*}
\frak M = \bigcup_{q\le W}\mathfrak M(q)\quad \text{and}\quad
\frak m = [0,1] \setminus\frak M,
\end{align*}
where $\mathfrak M(q) = \bigcup_{(a,q)=1}\{\alpha : |\alpha-a/q|\le 1/qQ_1\}$.
Recall the definitions \eqref{eq:PQj}, \eqref{eq:Sd},
\begin{align*}
[P_1,Q_1] &=  [(\log X)^{33A},(\log X)^{\psi(X)-4A}],\\
[P_2,Q_2] & = [\exp\big((\log X)^{2/3+\delta/2}\big),\,\exp\big((\log X)^{1-\delta/2}\big)],\\
\mathcal S_d(X,H,A,\delta) & = \{m\le X/d : \exists p_1,p_2\mid m \text{ with }p_j\in[P_j,Q_j]\}.
\end{align*}

We shall obtain Proposition \ref{prop:compmult} from the following results.
\begin{proposition}[Key minor arc estimate] \label{prop:minor}
Given any $A>5$, $H=(\log X)^{\psi(X)}$ with $40A \le \psi(X) \le (\log X)^{2/3}$, let $d\le W=(\log X)^A$ and $\mathcal S_d = \mathcal S_d(X,A,0)$ as in \eqref{eq:Sd}. Then for any completely multiplicative $g:\N\to\C$ with $|g|\le 1$, we have
\begin{align*}
\sup_{\substack{\alpha\in\,\frak{m}}} \int_0^{2X} \bigg|\sum_{\substack{x\le nd\le x + H\\n\in \mathcal S_d}}g(n)\,e(n\alpha)\bigg| \dd{x} \ \ll_A \ \frac{HX}{d^{3/4}W^{1/5}}.
\end{align*}
\end{proposition}

\begin{proposition}[Key major arc estimate for $\lambda$] \label{prop:major}
Given any $A>5$, $\delta>0$, $H=(\log X)^{\psi(X)}$ with $40A \le \psi(X) \le (\log X)^{2/3}$, let $d\le W=(\log X)^A$ and $\mathcal S_d = \mathcal S_d(X,H,A,\delta)$ as in \eqref{eq:Sd}. Then we have
\begin{align*}
\sup_{\substack{\alpha\in\,\frak{M}}} \int_0^{2X} \bigg|\sum_{\substack{x\le nd\le x + H\\n\in \mathcal S_d}}\lambda(n)\,e(n\alpha)\bigg| \dd{x} \ \ll_{A,\delta} \ \frac{HX}{dW}.
\end{align*}
\end{proposition}

We remark that the bounds in the minor arc hold for any bounded multiplicative function, whereas in the major arc the specific choice of $\lambda$ is needed.

\subsection{Minor arc}
In this subsection, we prove Proposition \ref{prop:minor}. Recall for $\alpha\in \frak{m}$ in the minor arc, $|\alpha-a/q|<W^4/qH$ with $q\in [W,H/W^4]$. It suffices to show
\begin{align}\label{eq:minor}
I_{\frak m} := \int_{\R} \theta(x)\sum_{x\le nd\le x + H}\1_{\mathcal S_d}(n) g(n)\,e(n\alpha) \dd{x} \ \ll \ HX \big(\tfrac{\log\log X}{dW}\big)^{1/2}\psi(X),
\end{align}
uniformly for any $\alpha\in \frak m$ and measurable $\theta:[0,2X]\to\C$ with $|\theta(x)|\le 1$. Letting $\mathcal P=\{p: P_1\le p\le Q_1\}$, by definition each $n\in \mathcal S_d$ has a prime factor in $\mathcal P$, so we use a variant of the Ramar\' e identity
\begin{align}
\1_{\mathcal S_d}(n) = \sum_{\substack{p\in\mathcal P\\n=mp}} \frac{\1_{\mathcal S^{(1)}_d}(mp)}{\#\{q\in \mathcal P : q\mid m\} + \1_{p\nmid m}},
\end{align}
where $\mathcal S^{(1)}_d=\{m\le X/d : \,\exists p\mid m, \, p\in[P_2,Q_2]\}$. As $g$ is completely multiplicative, we obtain
\begin{align*}
I_{\frak m} = \sum_{p\in\mathcal P}\sum_{m}\frac{\1_{\mathcal S^{(1)}_d}(mp) g(m)g(p)e(mp\alpha)}{\#\{q\in \mathcal P : q\mid m\} + \1_{p\nmid m}} \int_{\R} \theta(x)& \1_{x\le mpd\le x + H} \dd{x}.
\end{align*}
Next we split $\mathcal P$ into dyadic intervals $[P,2P]$.  It suffices to show for each $P\in [P_1,Q_1]$,
\begin{align}\label{eq:minordyad}
\sum_{\substack{p\in\mathcal P\\P\le p\le 2P}}\sum_{m}\frac{\1_{\mathcal S^{(1)}_d}(mp)g(m)g(p)e(mp\alpha)}{\#\{q\in \mathcal P : q\mid m\} + \1_{p\nmid m}} \int_{\R} \theta(x)& \1_{x\le mpd\le x + H} \dd{x} \ \ll \ \frac{HX}{\log P} \big(\tfrac{\log \log X}{dW}\big)^{1/2},
\end{align}
since then \eqref{eq:minor} will follow by \eqref{eq:PQj} and the triangle inequality, using
\begin{align*}
\sum_{\substack{P_1 \ll P\ll Q_1\\ P=2^j}} \frac{1}{\log P} \ll \sum_{\log P_1 \ll j\ll \log Q_1}\frac{1}{j} \ll \log\frac{\log Q_1}{\log P_1} = \log\frac{\psi(X)-4A}{33A} \ll_A \psi(X).
\end{align*}

Fix $P$. We may replace $\1_{p\nmid m}$ with 1 in \eqref{eq:minordyad} at a cost of $O(HX/dP)$. Indeed, since the integral is $\int_\R \theta(x) \1_{x\le mpd\le x + H} \dd{x} \ll H$, and $\1_{\mathcal S^{(1)}_d}(mp)=0$ unless $m\le X/dP$, the cost of such substitution is
\begin{align*}
\ll \ \sum_{\substack{p\in\mathcal P\\P\le p\le 2P}}\sum_{\substack{m\le X/dP\\p\mid m}} H \ \ll \ P\frac{X}{dP^2} H = \frac{HX}{dP}.
\end{align*}
Now the left hand side of \eqref{eq:minordyad} becomes
\begin{align*}
\sum_{m\in \mathcal S^{(1)}_d}&\frac{g(m)}{\#\{q\in \mathcal P : q\mid m\}+1} \sum_{\substack{p\in\mathcal P\\P\le p\le 2P}}g(p)e(mp\alpha)\1_{mpd\le X}\int_{\R} \theta(x) \1_{x\le mpd\le x + H} \dd{x}\\
& \ll \ \sum_{m\le X/dP}\bigg|\sum_{\substack{p\in\mathcal P\\P\le p\le 2P}}g(p)e(mp\alpha)\1_{mpd\le X}\int_{\R} \theta(x) \1_{x\le mpd\le x + H} \dd{x}\bigg|\\
& \ll \ (X/dP)^{1/2}\,\bigg(\sum_{m\le X/dP}\bigg|\sum_{\substack{p\in\mathcal P\\P\le p\le 2P}}g(p)e(mp\alpha)\1_{mpd\le X}\int_{\R} \theta(x) \1_{x\le mpd\le x + H}\dd{x} \bigg|^2\bigg)^{1/2},
\end{align*}
by the trivial bound and Cauchy--Schwarz. Hence for \eqref{eq:minordyad} it suffices to show
\begin{align}\label{eq:minCS}
\sum_{m\le X/dP}\bigg|\sum_{\substack{p\in\mathcal P\\P\le p\le 2P}}g(p)e(mp\alpha)\1_{mpd\le X}\int_{\R} \theta(x) \1_{x\le mpd\le x + H}\dd{x} \bigg|^2 \ \ll \ \frac{H^2PX}{W}\frac{\log\log P}{(\log P)^2}.
\end{align}
We expand the left hand side of \eqref{eq:minCS} and sum the resulting geometric series on $m$,
\begin{align*}
\underset{p_1,p_2\in\mathcal P\cap [P,2P]}{\sum\sum}\int_{\R^2}&g(p_1)\overline{g(p_2)} \theta(x_1)\overline{\theta(x_2)}\sum_{\substack{m\le X/dp_i \,\forall i\le2\\ x_i\le mdp_i\le x_i+H}}e\big(m(p_1-p_2)\alpha\big)\dd{x_1}\dd{x_2}\\
& \ll HX \sum_{p_1, p_2 \le 2P}\min\left(\frac{H}{dP}, \frac{1}{\|(p_1-p_2)\alpha\|}\right),
\end{align*}
since for given $d,p_1,p_2$, there are $O(X)$ choices for $x_1$ and $O(H)$ subsequent choices for  $x_2$ since $x_2 = x_1(p_2/p_1) + O(H)$.
Note $\|z\|$ denotes the distance of $z\in\R$ to the nearest integer.

Thus \eqref{eq:minCS} reduces to showing
\begin{align}\label{eq:minsieve}
\sum_{p_1, p_2 \le 2P}\min\left(\frac{H}{P}, \frac{1}{\|(p_1-p_2)\alpha\|}\right) \ \ll \ \frac{HP}{W}\frac{\log\log P}{(\log P)^2}.
\end{align}
The difference of primes is $p_1-p_2 \ll P$. Conversely, any integer $n\ll P$ may be written as $n=p_1-p_2$ for $p_1,p_2\le\, 2P$ in $\ll \frac{n}{\phi(n)}P(\log P)^{-2} \ \ll P \frac{\log\log P}{(\log P)^2}$ ways by a standard upper bound sieve, see \cite[Proposition 6.22]{Opera}. Hence for \eqref{eq:minsieve} it suffices to obtain
\begin{align*}
\sum_{1\le n \ll P}\min\left(\frac{H}{n}, \frac1{\|n\alpha\|}\right)  \ \ll \ \frac{H}{W} \qquad (\alpha\in \,\frak{m}).
\end{align*}
But this follows by the standard `Vinogradov lemma' \cite[p.346]{IK}.
\begin{lemma}\label{lem:vino}
Given $H,P>1$, take $\alpha\in[0,1]$ with $|\alpha-a/q|\le 1/q^2$ for some $(a,q)=1$. Then
\begin{align*}
\sum_{1\le n \le P}\min\left(\frac{H}{n}, \frac1{\|n\alpha\|}\right) \ \ll \ \frac{H}{q} + \frac{H}{P} + (P+q)\log q.
\end{align*}
\end{lemma}
Observe $H/q + H/P + (P+q)\log q \, \ll \, H/W$ since $q\in [W,H/W^4]$, $P\in [P_1,Q_1]=[W^{24},H/W^4]$.  This completes the proof in the minor arc.

\subsection{Major arc}
In this subsection, we prove the key major arc estimate assuming the following mean value result for the (twisted) Liouville function.
\begin{proposition} \label{prop:J}
Given $A>5$, $\delta>0$, $H=(\log X)^{\psi(X)}$ with $40A \le \psi(X) \le (\log X)^{2/3}$, let $q\le W=(\log X)^A$, $d<W^{33}$, $\chi$ \textnormal{(mod $q$)}, $h\in [H/W^5,H]$, and $\mathcal S_d=\mathcal S_d(X,H,A,\delta)$ as in \eqref{eq:Sd}. Then for all $Y\in [X/W^7,2X]$, we have
\begin{align}
J_{d,h,q}(Y;\chi):=\int_Y^{2Y}\bigg|\frac{1}{h}\sum_{\substack{x\le m\le x + h\\m\in \mathcal S_d}} \lambda(m)\chi(m)\bigg|^2\dd{x} \ \ll_{A,\delta} \  \frac{Y}{W^{10}}.
\end{align}
\end{proposition}

\begin{proof}[Proof of Proposition \ref{prop:major} from Proposition \ref{prop:J}]
To obtain the key major arc estimate we shall prove the stronger bound,
\begin{align}\label{eq:major}
I_{\frak M} := \sup_{\alpha\in\frak M}\int_0^{2X} \bigg|\sum_{x\le nd\le x + H}\1_{\mathcal S_d}(n)\lambda(n)\,e(n\alpha)\bigg| \dd{x} \ \ll \  \frac{HX}{dW}.
\end{align}

In the major arc recall $\alpha = \frac{a}{q} + \theta$ with $q\le W$ and $|\theta| \le \frac{W^4}{qH}$. By partial summation with $a_n = \1_{>x/d}(n)\1_{\mathcal S_d}(n) \lambda(n)\,e(na/q)$, and $A(t) = \sum_{n\le t} a_n$, we have
\begin{align*}
\sum_{x\le nd\le x + H} &\1_{\mathcal S_d}(n) \lambda(n)\,e(n\alpha) = e(\tfrac{x+H}{d}\theta)A(\tfrac{x+H}{d}) - e(\tfrac{x}{d}\theta)A(\tfrac{x}{d}) - 2\pi i\theta\int_{x/d}^{(x+H)/d} e(t\theta)A(t)\dd{t} \\
\ll \ & \bigg|\sum_{x\le nd\le x + H}\1_{\mathcal S_d}(n)\lambda(n)\,e(an/q)\bigg| \ + \ |\theta|\int_0^{H/d}\bigg|\sum_{x/d\le n\le x/d + h}\1_{\mathcal S_d}(n)\lambda(n)\,e(na/q)\bigg|\dd{h}.
\end{align*}
Thus taking the maximizing $h$ and integrating over $x\in[0,2X]$, we obtain
\begin{align}\label{eq:ImajIh}
I_{\frak M} \ \ll \  I_{H/d} \ + \  |\theta| \frac{H}{d}\sup_{h\le H/d}I_h \ \ll \  I_{H/d} \ + \ \frac{W^4}{qd}\sup_{h\le H/d}I_h,
\end{align}
where
\begin{align}\label{eq:ImajI}
I_h & :=\int_0^{2X} \bigg|\sum_{x/d\le n\le x/d + h}\1_{\mathcal S_d}(n)\lambda(n) \,e(an/q)\bigg|\dd{x}.
\end{align}
Then splitting into residues $b$ (mod $q$) gives
\begin{align*}
I_h \le \sum_{b\,(q)}|e(ab/q)| \int_0^{2X} \bigg|\sum_{\substack{x/d\le n\le x/d + h\\n\equiv b\,(q)}}\1_{\mathcal S_d}(n)\lambda(n)\bigg|\dd{x} \ = \ \sum_{b\,(q)}\int_0^{2X} \bigg|\sum_{\substack{x/d\le n\le x/d + h\\n\equiv b\,(q)}}\1_{\mathcal S_d}(n)\lambda(n)\bigg|\dd{x}.
\end{align*}

Now suppose we have the bound
\begin{align}\label{eq:Ihbound}
I_h \ \ll \ \frac{qhX}{W^5}  \qquad\quad\text{for}\qquad h\in[qH/W^5,H/d].
\end{align}
Then, combining with the trivial bound $I_h\le hX$ when $h\le qH/W^5$, \eqref{eq:ImajIh} becomes
\begin{align*}
I_{\frak M} \ &\ll \ I_{H/d} \quad + \quad \frac{W^4}{qd}\Big(\sup_{qH/W^5 \le h\le H/d}I_h \ + \sup_{h\le qH/W^5}hX\Big)\\
& \ll \ \frac{qHX}{dW^5} \quad + \quad \frac{W^4}{qd}\Big(\frac{qHX}{dW^5} \ + \ \frac{qHX}{W^5}\Big) \ \ll \quad \frac{HX}{dW},
\end{align*}
for $q\le W$ in the major arc. Hence it suffices to show \eqref{eq:Ihbound}.

Now to bound $I_h$, we extract the gcd. Let $c := (b,q)$ so that $c\mid n$, and we let $b' = b/c$, $q'=q/c$, $h'=h/c$, $m=n/c$. Thus since $\lambda$ is completely multiplicative, we have
\begin{align*}
I_h \ & \le \sum_{c\mid q}|\lambda(c)|\sideset{}{^*}\sum_{b'\,(q')} \int_0^{2X}\bigg|\sum_{\substack{x/cd\le m\le x/cd + h/c\\m\equiv b'\,(q')}}\1_{\mathcal S_d}(cm)\lambda(m)\bigg|\dd{x}\\
& \le \sum_{c\mid q}cd\sideset{}{^*}\sum_{b'\,(q')} \int_0^{2X/cd}\bigg|\sum_{\substack{y\le m\le y + h'\\m\equiv b'\,(q')}}\1_{\mathcal S_{cd}}(m)\lambda(m)\bigg|\dd{y},
\end{align*}
using the substitution $y=x/cd$, and noting $\1_{\mathcal S_d}(cm) = \1_{\mathcal S_{cd}}(m)$ since $c\le q<P_1$. Then recalling orthogonality of characters $\phi(q')\1_{m\equiv b'\,(q')} \ = \ \sum_{\chi\,(q')}\overline{\chi(b')}\chi(m)$, we obtain
\begin{align}
I_h \ &\le \sum_{c\mid q}cd\sideset{}{^*}\sum_{b'\,(q')}\frac{1}{\phi(q')}\sum_{\chi\,(q')}|\overline{\chi(b')}| \int_0^{2X/cd}\bigg|\sum_{y\le m\le y + h'}\1_{\mathcal S_{cd}}(m)\lambda(m)\chi(m)\bigg|\dd{y} \nonumber\\
&\le \sum_{c\mid q}cd \sum_{\chi\,(q')} \int_0^{2X/cd}\bigg|\sum_{y\le m\le y + h'}\1_{\mathcal S_{cd}}(m)\lambda(m)\chi(m)\bigg|\dd{y}. \label{eq:majII}
\end{align}
We may discard the contribution to \eqref{eq:majII} of the integral over $y\le X/dW^5$, since $h'=h/c$ and $q\le W$ imply an admissible cost
\begin{align*}
\ll \ \sum_{c\mid q}cd\phi(q') \, \frac{X}{dW^5} h'
\ \le \ \frac{hX}{W^5} \sum_{q'\mid q}\phi(q')
\ = \ \frac{qhX}{W^5}.
\end{align*}
For the remaining $y\in [\frac{X}{dW^5}, \frac{X}{cd}]$ in  \eqref{eq:majII}, we split into dyadic intervals so that
\begin{align}\label{eq:ItoMR}
I_h & \ \le \sum_{c\mid q}cd \sum_{\chi\,(q')}\sum_{\substack{Y = 2^j\\\frac{X}{2dW^5}\le Y\le \frac{2X}{cd}}}\int_Y^{2Y}\bigg|\sum_{y\le m\le y + h'}\1_{\mathcal S_{cd}}(m)\lambda(m)\chi(m)\bigg|\dd{y} \ + \ O\Big(\frac{qHX}{W^5}\Big).
\end{align}
By assumption, Proposition \ref{prop:J} implies $J_{cd,h',q'}(Y;\chi) \ll Y/W^{10}$, so Cauchy--Schwarz gives
\begin{align*}
\int_Y^{2Y}\bigg|\sum_{y\le m\le y + h'}\1_{\mathcal S_{cd}}(m) \lambda(m)\chi(m)\bigg|\dd{y} \le h'\,\sqrt{Y\cdot J_{cd,h',q'}(Y;\chi)} \ \ll_A \frac{Yh'}{W^5}.
\end{align*}
So plugging back into \eqref{eq:ItoMR}, we obtain
\begin{align*}
I_h & \ \ll \ \sum_{c\mid q}cd \phi(q')\sum_{\substack{Y = 2^j\\\frac{X}{2W^4}\le Y\le \frac{2X}{cd}}} \frac{Yh'}{W^5} \ = \ \frac{hd}{W^5}\sum_{q'\mid q} \phi(q')\sum_{\substack{Y = 2^j\\\frac{X}{2dW^5}<Y< \frac{2X}{cd}}} Y  \ \ll \  \frac{qhX}{W^5}.
\end{align*}
This gives \eqref{eq:Ihbound} as desired.
\end{proof}

\section{Preparatory lemmas} 

We collect some standard lemmas on Dirichlet polynomials.

The first is the integral mean value theorem \cite[Theorem 9.1]{{IK}}.
\begin{lemma}[mean value] \label{lem:meanval}
For $D(s) = \sum_{n\le N} a_n n^{-s}$, we have
\begin{align*}
\int_{-T}^T |D(it)|^2\dd{t} = (T+O(N))\sum_{n\le N}|a_n|^2.
\end{align*}
\end{lemma}

One may discretize the mean value theorem by replacing the intergal over $[-T,T]$ with a sum over a well-spaced set $\mathcal W\subset [-T,T]$.

\begin{definition}
A set $\mathcal W\subset \R$ is {\it well-spaced} if $|w-w'|\ge1$ for all $w,w'\in \mathcal W$.
\end{definition}

Next is the Hal\'asz-Montgomery inequality \cite[Theorem 9.6]{{IK}}, which offers an improvement to the (discretized) mean value theorem when the well-spaced set is `sparse.'

\begin{lemma}[Hal\'asz-Montgomery]\label{lem:HalMont}
Given $D(s) = \sum_{n\le N} a_n n^{-s}$ and a well-spaced set $\mathcal W\subset [-T,T]$. Then
\begin{align*}
\sum_{t\in \mathcal W}|D(it)|^2 \ \ll \ (N + |\mathcal W|\sqrt{T})\log 2T \sum_{n\le N}|a_n|^2.
\end{align*}
\end{lemma}

We also need a bound on the size of well-spaced sets $\mathcal W$ in terms of the values of prime Dirichlet polynomials on $1+i\mathcal W$ \cite[Lemma 8]{MR}.
\begin{lemma} \label{lem:MR8}
Let $a_p\in \C$ be indexed by primes, with $|a_p|\le 1$, and define the prime polynomial
\begin{align*}
P(s) = \sum_{L\le p\le 2L}\frac{a_p}{p^s}.
\end{align*}
Suppose a well-spaced set $\mathcal W\subset [-T,T]$ satisfies $|P(1+it)|\ge1/U$ for all $t\in\mathcal W$. Then
\begin{align*}
|\mathcal W|  \ \ll \ U^2 \,T^{2(\log U + \log\log T)/\log L}.
\end{align*}
\end{lemma}

\begin{lemma} \label{lem:MR2l}
Given $A,K>0$, $\theta > \frac{2}{3}$, and a Dirichlet character $\chi$ mod $q\le (\log X)^A$. Assume $\exp((\log X)^\theta) \le P\le Q\le X$, and let $P(s,\chi) = \sum_{P\le p\le Q}\chi(p)p^{-s}$. Then for any $|t|\le X$,
\begin{align*}
|P(1+it,\chi)| \ \ll_{A,K,\theta} \ \frac{\log X}{1+|t|} \ + \ (\log X)^{-K}.
\end{align*}
\end{lemma}
\begin{proof}
This follows as with \cite[Lemma 2]{MRshort}, except that the Vinogradov--Korobov zero-free region for $\zeta(s)$ is replaced by that of $L(s,\chi)$.
\end{proof}

We also use a Parseval-type bound. This shows that the average of a multiplicative function in almost all short intervals can be approximated by its average on a long interval, provided the mean square of the corresponding Dirichlet polynomial is small.
\begin{lemma}[Parseval bound] \label{lem:MR14}
Given $T_0 \in [(\log X)^{1/15}, X^{1/4}]$, and take a sequence $(a_m)_{m=1}^\infty$ with $|a_m|\le 1$. Assume $1\le h_1\le h_2 \le X/T_0^3$. For $x\in[X,2X]$, define
\begin{align*}
S_{h_j}(x) = \sum_{x\le m\le x+h_j}a_m, \qquad \textnormal{and}\qquad A(s) = \sum_{X\le m\le 4X}\frac{a_m}{m^s}. 
\end{align*}
Then
\begin{align*}
\frac{1}{X}\int_X^{2X}\Big|\tfrac{1}{h_1}S_{h_1}(x)-\tfrac{1}{h_2}S_{h_2}(x)\Big|^2\dd{x} \ \ll \ & \frac{1}{T_0} \ + \ \int_{T_0}^{X/h_1} |A(1+it)|^2\dd{t}\\
& \ + \max_{T\ge X/h_1}\frac{X/h_1}{T}\int_T^{2T} |A(1+it)|^2\dd{t}.
\end{align*}
\end{lemma}
\begin{proof}
This follows as in \cite[Lemma 14]{MR} with $(\log X)^{1/15}$ replaced by general $T_0$.
\end{proof}

We have a general mean value of products, via the Ramar\'e identity \cite[Lemma 12]{MR}.
\begin{lemma} \label{lem:MR12}
For $V,P,Q\ge1$, denote $\mathcal P=[P,Q]\cap \mathbb P$. Let $a_m,b_m,c_p$ be bounded sequences for which $a_{mq}=b_mc_q$ when $q\nmid m$ and $q\in \mathcal P$. Let
\begin{align*}
Q_{v,V}(s) &= \sum_{\substack{q\in \mathcal P\\ e^{v/V}\le q\le e^{(v+1)/V}}}\frac{c_q}{q^s},\\
R_{v,V}(s) &= \sum_{Xe^{-v/V}\le m\le 2Xe^{-v/V}}\frac{b_m}{m^s}\cdot \frac{1}{\#\{p\mid m : p\in \mathcal P\}+1},
\end{align*}
and take a measurable set $\mathcal T\subset [-T,T]$. Then for $\mathcal I=[\lfloor V \log P\rfloor,V\log Q]\cap \Z$, we have
\begin{align*}
\int_{\mathcal T}\Big|\sum_{X\le n \le 2X}\frac{a_n}{n^{1+it}}\Big|^2\dd{t} \ \ll \ V\log(\tfrac{Q}{P}) & \sum_{v\in \mathcal I}\int_{\mathcal T}|Q_{v,V}(1+it)\, R_{v,V}(1+it)|^2 \dd{t}\\
& \qquad \ + \ \Big(\frac{T}{X}+1\Big)\Big(\frac{1}{V}+\frac{1}{P}+\sum_{\substack{X\le n\le 2X\\p\nmid n\forall p\in \mathcal P}}\frac{|a_n|^2}{n}\Big)
\end{align*}
\end{lemma}

In the next result we employ the Fundamental Lemma of the sieve, along with the Siegel--Walfisz theorem.

\begin{lemma}\label{lem:SxS'}
Given $A,K>0$, $q\le (\log x)^A$, Dirichlet character $\chi$ \textnormal{(mod $q$)}, and let $\mathcal D = \prod_{p\in \mathcal P}p$ for any set of primes $\mathcal P\subset (q,x^{1/\log\log x})$. Then
\begin{align*}
\sum_{\substack{m\le x\\(m,\mathcal D)=1}}\lambda(m)\chi(m) \ \ll_{A,K} \ \frac{x}{(\log x)^K}.
\end{align*}
\end{lemma}
\begin{proof}
First partition the sum on $m$ by the values of $\lambda(m),\chi(m)$,
\begin{align}\label{eq:mlamchi}
S_0 := \sum_{m\le x}\1_{(m,\mathcal D)=1}\lambda(m)\chi(m) = \sum_{b\,(q), \nu\in\{\pm1\}}\nu\chi(b)\sum_{m\in \mathcal A^{(b,\nu)}}\1_{(m,\mathcal D)=1}
\end{align}
for the set $\mathcal A^{(b,\nu)} = \{m\le x : m\equiv b\,(q), \lambda(m)=\nu\}$.

Now it suffices to prove
\begin{align}\label{eq:SAP}
\sum_{m\in \mathcal A^{(b,\nu)}}\1_{(m,\mathcal D)=1} \ = \ \frac{x}{2 q}\prod_{p\mid \mathcal D}\Big(1-\frac{1}{p}\Big) \ + \  O_{A,K}\big(x(\log x)^{-K}\big)
\end{align}
uniformly in $b,\nu$, from which it will follow
\begin{align*}
S_0 = \frac{x}{2 qd}\prod_{p\mid \mathcal D}\Big(1-\frac{1}{p}\Big)\sum_{b\,(q), \nu\in\{\pm1\}}\nu\chi(b) \ + \  O\big(x(\log x)^{-K}\big) \ \ll \ 
\frac{x}{(\log x)^{K}},
\end{align*}
by pairing up terms $\nu=\pm1$. This will give the lemma.

Now to show \eqref{eq:SAP}, write $\mathcal A = \mathcal A^{(b,\nu)}$. For $d\mid \mathcal D$ the set of multiples $\mathcal A_d=\{m\in \mathcal A: d\mid m\}$ has size
\begin{align*}
|\mathcal A_d| = \sum_{\substack{m\le x\\d\mid m,\,m\equiv b\;(q)}} \1_{\lambda(m)=\nu} 
= \sum_{\substack{n\le x/d\\nd\equiv b\;(q)}}\frac{\nu\lambda(nd)+1}{2}
= \frac{\nu\lambda(d)}{2}\sum_{\substack{n\le x/d\\n\equiv bd^{-1}\;(q)}}\lambda(n) \ + \ \frac{x}{2qd} + O(q)
\end{align*}
noting $(q,d)=1=(q, \mathcal D)$. Moreover $\max_{c\,(q)}\big|\sum^{n\le y}_{n\equiv c\,(q)}\lambda(n)\big| \ll_{A,K} y(\log y)^{-2K}$ by Siegel--Walfisz, which is valid by the assumption $q\le (\log X)^A$. Thus
\begin{align}\label{eq:AdSW}
|\mathcal A_d| = \frac{x}{2qd} + R_d, \qquad\qquad \text{where}\quad |R_d|\ll \frac{x}{d}\log(x/d)^{-2K} + q.
\end{align}

Now for any $D>1$ the indicator $\1_{(m,\mathcal D)=1}$ is bounded in between $\sum_{d<D}^{d\mid (m, \mathcal D)} \lambda^\pm_d$, for the standard linear sieve weights $\{\lambda^\pm_d\}_{d<D}$, see \cite[Lemma 6.11]{Opera}. Thus the desired sum in \eqref{eq:SAP} is bounded in between
\begin{align}
\sum_{\substack{d\mid \mathcal D\\d<D}}\lambda_d^- |\mathcal A_d| \ \le \sum_{m\in \mathcal A}\1_{(m, \mathcal D)=1} \ \le \
\sum_{\substack{d\mid \mathcal D\\d<D}}\lambda_d^+ |\mathcal A_d|.
\end{align}
Note by \eqref{eq:AdSW}, the upper and lower bounds are given by 
\begin{align*}
\sum_{\substack{d\mid \mathcal D\\d<D}}\lambda_d^\pm |\mathcal A_d| = \frac{x}{2 q}\sum_{\substack{d\mid \mathcal D\\d<D}}\frac{\lambda_d^\pm}{d} \ + \ \sum_{\substack{d\mid \mathcal D\\d<D}}\lambda_d^\pm R_d.
\end{align*}
Let $z=x^{1/\log_2 x}$ so that $\mathcal P\subset (q,z)$. Then choosing $D = z^s$ for $s = 2K\log_2 x/\log_3 x$, the above error is $\ll\sum_{d<D}|R_d| \ll x(\log x)^{-K}$ by \eqref{eq:AdSW}. And by the Fundamental Lemma \cite[Lemma 6.11]{Opera}, the main term is
\begin{align*}
\sum_{\substack{d\mid \mathcal D\\d<D}}\frac{\lambda_d^\pm}{d} = (1+O(s^{-s}))\prod_{p\mid \mathcal D}\Big(1-\frac{1}{p}\Big).
\end{align*}
Hence \eqref{eq:SAP} follows as claimed, noting $s^{-s}\ll (\log x)^{-K}$.
\end{proof}

\section{Mean value of multiplicative functions}

In this section we prove Proposition \ref{prop:J} based on the following mean value theorem for Dirichlet polynomials with typical factorization. This refines Matom\"aki--Radziwi\l\l \ \cite[Proposition  12]{MR} in the case of $g=\lambda\,\chi$, by leveraging Vinogradov--Korobov type bounds.

\begin{proposition}\label{prop:MRmain}
Given any $A>5$, $\delta>0$, denote $B=11A$, and write $W=(\log X)^A$, $H=(\log X)^{\psi(X)}$ with $40A \le \psi(X) \le (\log X)^{2/3}$. Take $q\le W$, $d<W^{33}$, a Dirichlet character $\chi$ \textnormal{(mod $q$)}, and let $\mathcal S_d=\mathcal S_d(X,H,A,\delta)$ as in \eqref{eq:Sd} (as such $[P_1,Q_1]=[W^{33},H/W^4]$). For any $Y\in[X^{1/2},X^2]$, define
\begin{align*}
G(s) = \sum_{\substack{Y\le n\le 2Y\\ n\in \mathcal S_d}}\frac{\lambda(n)\chi(n)}{n^s}.
\end{align*}
Then for any $T\in [Y^{1/2},Y]$, we have
\begin{align}
\int_{(\log X)^{2B}}^T |G(1+it)|^2\dd{t} \ \ll_{A,\delta} \ \Big(\frac{Q_1T}{Y}\ + \ 1\Big)(\log X)^{-B}.
\end{align}
\end{proposition}

\begin{proof}[Proof of Proposition \ref{prop:J} from Proposition \ref{prop:MRmain}]
We shall prove
\begin{align}
J:=\int_Y^{2Y}\Big|\tfrac{1}{h_1}S_{h_1}(x)\Big|^2\dd{x} \ \ll_{A,\delta} \ \frac{Y}{(\log X)^{10A}}
\end{align}
for $Y\in [X/W^7,2X]$, $h_1=h\in [qH/W^5,H]$ and the sum
\begin{align*}
S_l(x) := \sum_{\substack{x\le m\le x+l\\m\in\mathcal S_d}}\lambda(m)\chi(m).
\end{align*}

First we claim $S_x(0) \ll x(\log x)^{-K}$ for all $K>0$. To this, recall from \eqref{eq:Sd} that each $m\in \mathcal S_d$ has prime factors $p_1,p_2\mid m$ with $p_j\in[P_j,Q_j]$. So by inclusion-exclusion, the indicator of $\mathcal S_d$ is
\begin{align*}
\1_{\mathcal S_d}(m) &= (1 - \1_{(m,\mathcal D_1)=1})(1-\1_{(m,\mathcal D_2)=1})\\
&=\1_{(m,\mathcal D_0)=1} - \1_{(m,\mathcal D_1)=1} -  \1_{(m,\mathcal D_2)=1} +  \1_{(m,\mathcal D_3)=1},
\end{align*}
where $\mathcal D_j=\prod_{p\in \mathcal P_j}p$ for the sets of primes $\mathcal P_0=\emptyset$, $\mathcal P_1=[P_1,Q_1]$, $\mathcal P_2=[P_2,Q_2]$, $\mathcal P_3=\mathcal P_1\cup\mathcal P_2$. Hence applying Lemma \ref{lem:SxS'} to each $\mathcal D_j$ gives
\begin{align}
S_x(0) & \le \sum_{j=0}^3 \bigg|\sum_{m\le x}\1_{(m,\mathcal D_j)=1}\lambda(m)\chi(m)\bigg|  \ \ll_{A,K} \ x(\log x)^{-K}.
\end{align}

In particular, letting $B = 11A$ we have
\begin{align*}
S_{h_2}(x) = S_{x+h_2}(0) - S_x(0) \ll_A \frac{x}{(\log x)^{7B}} \ll_A \frac{h_2}{(\log x)^B},
\end{align*}
where $h_2 \asymp x(\log x)^{-6B}$, and so
\begin{align}
J=\frac{1}{Y}\int_Y^{2Y}\Big|\tfrac{1}{h_1}S_{h_1}(x)\Big|^2\dd{x}  \ll  (\log X)^{-B} \ + \frac{1}{Y}\int_Y^{2Y}\bigg|\tfrac{1}{h_1}S_{h_1}(x) \ - \ \tfrac{1}{h_2}S_{h_2}(x)\bigg|^2\dd{x}.
\end{align}
Then Lemma \ref{lem:MR14} (Parseval) with $T_0=(\log X)^{2B}$ and $h_2 = Y/T_0^3 = Y/(\log X)^{6B}$ gives
\begin{align}\label{eq:Jparseval}
J & \ll (\log X)^{-B} + \int_{T_0}^{Y/h_1} |G(1+it)|^2\dd{t} \ + \max_{T\ge Y/h_1}\frac{Y/h_1}{T}\int_T^{2T} |G(1+it)|^2\dd{t}.
\end{align}
Now for the latter integral over $[T,2T]$, we apply the Lemma \ref{lem:meanval} (mean value) if $T\ge X/2$, and apply Proposition \ref{prop:MRmain} if $T\in[Y/h_1,X/2]$. Doing so, \eqref{eq:Jparseval} becomes
\begin{align}\label{eq:Jbounded}
J & \ll \big(\tfrac{Y/h_1}{Y/Q_1}+1\big)(\log X)^{-B}
 \ + \ \max_{T\ge X/2}\frac{Y/h_1}{T}(T/X+1) \nonumber\\
 & \qquad\qquad\qquad\qquad\ + \max_{Y/h_1 \le T\le X/2} \frac{Y/h_1}{T}\big(\tfrac{Q_1T}{Y}+1\big)(\log X)^{-B} \nonumber\\
 & \ \ll \ (\frac{Q_1}{h_1}+1)(\log X)^{-B} + \frac{Y}{h_1 X}  
 \ \ll \ W(\log X)^{-B} \ = \ (\log X)^{A-B}.
\end{align}
Here we used $Y\le\, 2X$, $Q_1 = H/W^4$, and $h_1 \ge \,H/W^5 \,(\ge\, W^{35})$. Hence recalling $B=11A$ gives Proposition \ref{prop:J} as claimed.
\end{proof}

\subsection{Mean value of Dirichlet polynomials} \label{sec:MR}

In this subsection, we prove Proposition \ref{prop:MRmain}. Recall the definitions \eqref{eq:PQj}, \eqref{eq:Sd},
\begin{align*}
[P_1,Q_1] &=  [(\log X)^{33A},(\log X)^{\psi(X)-4A}],\\
[P_2,Q_2] & = [\exp\big((\log X)^{2/3+\delta/2}\big),\,\exp\big((\log X)^{1-\delta/2}\big)],\\
\mathcal S_d(X,H,A,\delta) & = \{m\le X/d : \exists p_1,p_2\mid m \text{ with }p_j\in[P_j,Q_j]\}.
\end{align*}

Let $B=11A$, and $\alpha = 1/5$. Let $V = P_1^{1/3}=(\log X)^B$ and define the prime polynomial
\begin{align}
Q_{v,j}(s) := \sum_{\substack{P_j\le p\le Q_j\\e^{v/V}\le p\le e^{(v+1)/V}}}\frac{\lambda(p)\chi(p)}{p^s}.
\end{align}
Note $Q_{v,j}(s)\neq0$ only if $v\in \mathcal I_j:= \{v\in\Z: P_j\le e^{v/V} \le Q_j\} =[\lfloor V \log P_j\rfloor, V\log Q_j]$.

We decompose $[T_0,T] = \mathcal T_1\cup \mathcal T_2$ as a disjoint union, where $\mathcal T_2=[0,1]\setminus \mathcal T_1$ and
\begin{align}
\mathcal T_1 = \{t : \ |Q_{v,1}(1+it)| \le e^{-\alpha v/V} \quad \forall v\in\mathcal I_1 \}.
\end{align}
For $j=1,2$ denote by $\mathcal S_d^{(j)}$ the integers containing a prime factor in the interval $[P_i,Q_i]$ with $i\neq j$ and possibly, but not necessarily, with $i=j$. That is,
\begin{align*}
\mathcal S_d^{(j)} & = \{m\le X/d : \exists p\mid m \text{ with }p\in[P_i,Q_i] \text{ for }i\neq j\}.
\end{align*}
Also define the polynomial
\begin{align}
R_{v,j}(s) = \sum_{\substack{Ye^{-v/V}\le m\le 2Ye^{-v/V}\\m\in\mathcal S_d^{(j)}}}\frac{\lambda(m)\chi(m)}{m^s}\cdot \frac{1}{\#\{p\mid m : P_j\le p\le Q_j\}+1}.
\end{align}

Now Lemma \ref{lem:MR12} (Ramar\'e) applies with $V = V, P=P_j, Q=Q_j$, and $a_m=\lambda\,\chi\,\1_{\mathcal S}(m)$, $c_p=\lambda\,\chi(p)$, $b_m=\lambda\,\chi\,\1_{\mathcal S_j}(m)$, giving
\begin{align*}
\int_{\mathcal T_j}|G(1+it)|^2\dd{t} \  \ll \ V\log Q_j \sum_{v\in \mathcal I_j} & \int_{\mathcal T_j} |Q_{v,j}(1+it)\,R_{v,j}(1+it)|^2 \dd{t}\\
&  \ + \ \frac{1}{V} + \frac{1}{P_j} \ + 
\sum_{\substack{Y\le n\le 2Y\\ p\nmid n\forall p\in [P_j,Q_j]}} \frac{\1_{\mathcal S_d}(n)}{n}.
\end{align*}
We crucially note the latter sum vanishes since each $n\in \mathcal S_d$ has a prime factor $p\in[P_j,Q_j]$. Summing over $j=1,2$, the second and third terms above contribute
\begin{align*}
\ll \sum_{1\le j\le 2}\Big(\frac{1}{V} + \frac{1}{P_j}\Big) \ll \frac{1}{V} = (\log X)^{-B}.
\end{align*}
Hence the desired integral is
\begin{align}
\int_{T_0}^T|G(1+it)|^2\dd{t} = \int_{\mathcal T_1\cup \mathcal T_2} |G(1+it)|^2\dd{t} \ &\ll \ E_1 + E_2 + (\log X)^{-B},
\end{align}
where
\begin{align}
E_j & = V\log Q_j \sum_{v\in\mathcal I_j} \int_{\mathcal T_j}|Q_{v,j}(1+it)\,R_{v,j}(1+it)|^2 \dd{t}.
\end{align}
Hence it suffices to bound $E_1,E_2\ll (Q_1T/Y+1)(\log X)^{-B}$.

{\bf Bound for $E_1$:}
By definition of $t\in\mathcal T_1$, we have $|Q_{v,1}(1+it)| \le e^{-\alpha v/V}$ for all $v\in\mathcal I_1$, so
\begin{align*}
E_1 \ \ll \ V\log Q_1 \sum_{v\in\mathcal I_1} e^{-2\alpha v/V}\int_{\mathcal T_1}|R_{v,1}(1+it)|^2 \dd{t}
\ \ll \ V\log Q_1 \sum_{v\in\mathcal I_1} e^{-2\alpha v/V} \Big(\frac{T}{Y/e^{v/V}}+1\Big)
\end{align*}
by Lemma \ref{lem:meanval} (mean value). Summing the resulting geometric series gives
\begin{align}
E_1  \ &\ll \ V\log Q_1 \frac{P_1^{-2\alpha}}{1-e^{-2\alpha/V}} \Big(\frac{Q_1T}{Y}+1\Big) \nonumber\\
& \ll \ (\log X)P_1^{-2\alpha}\Big(\frac{Q_1T}{Y}+1\Big)
\ \ll \ (\log X)^{-B}\Big(\frac{Q_1T}{Y}+1\Big),
\end{align}
noting $V/(1-e^{-2\alpha/V}) = O(1)$ and $P_1^{-2\alpha}=(\log X)^{-6B/5}$.

{\bf Bound for $E_2$:}
We choose the maximizing $v\in \mathcal I_2$ for $E_2$. Thus since $|\mathcal I_2|< V\log Q_2$,
\begin{align*}
E_2 & = V\log Q_2 \sum_{v\in\mathcal I_2} \int_{\mathcal T_2} |Q_{v,2}\cdot R_{v,2}(1+it)|^2 \dd{t}
 \ll (V\log Q_2)^2 \int_{\mathcal T_2}|Q_{v,2}\cdot R_{v,2}(1+it)|^2 \dd{t} \nonumber\\
 & \ \le (V\log Q_2)^2  \sum_n \sup_{t_n\in [n,n+1]\cap\, \mathcal T_2}|Q_{v,2}\cdot R_{v,2}(1+it_n)|^2 \nonumber\\
 \ & \ \le 2(V\log Q_2)^2\sum_{t\in \mathcal W}|Q_{v,2}\cdot R_{v,2}(1+it)|^2,
\end{align*}
for a well-spaced set $\mathcal W \subset \mathcal T_2$. For instance, one may take $\mathcal W$ as the even or odd integers in $\mathcal T_2$ (choose the parity that gives a larger contribution). It happens that $\mathcal W$ is easier to analyze than $\mathcal T_2$ itself.

Now is the critical step for the choice $g=\lambda\chi$ and $\log P_2 = (\log X)^{2/3+\delta}$: by Lemma \ref{lem:MR2l} (Vinogradov--Korobov), we have for all $t\in [T_0,T]$
\begin{align}
|Q_{v,2}(1+it)| \ \ll_{\delta,A} \ \frac{\log X}{1+T_0} + (\log X)^{-B} \ll (\log X)^{-B},
\end{align}
for $T_0=(\log X)^{2B}$ and $B=11A$.
So by Lemma \ref{lem:HalMont} (Hal\'asz-Montgomery), we have
\begin{align*}
E_2 \ & \ll \ (V\log Q_2)^2(\log X)^{2-4B} \sum_{t\in \mathcal W} |R_{v,2}(1+it)|^2 \\
& \ll  (V\log Q_2)^2(\log X)^{3-4B}(Ye^{-v/V} + |\mathcal W|\sqrt{T}) \, \frac{e^{v/V}}{Y} \\
&  \ \ll \ (\log X)^{-B}(1 + |\mathcal W| \frac{\sqrt{T} Q_2}{Y}),
\end{align*}
recalling $\log T\asymp \log X$, $V = (\log X)^B$, and $e^{v/V}\le Q_2$.

Thus it suffices to bound $|\mathcal W|$. We shall obtain
\begin{align}
E_2 \ \ll \ (\log X)^{-B}\Big(1+\frac{T}{Y}\Big) \ \ll_{A,\delta} \ (\log X)^{-B}
\end{align}
provided we show $|\mathcal W| \ll T^{1/2}/Q_2$. To prove this, by definition of $\mathcal T_2 \supset\mathcal W$, we first partition $\mathcal W = \bigcup_{u\in \mathcal I_1}\mathcal W^{(u)}$ where
\begin{align*}
|Q_{u,1}(1+it)|> e^{-u\alpha/V} \qquad\text{for all}\quad t\in \mathcal W^{(u)}.
\end{align*}
Hence for each $u\in\mathcal I_1$, we may apply Lemma \ref{lem:MR8} to the prime polynomial $Q_{u,1}$ with $U = e^{u\alpha/V}$ and $L = e^{u/V}$, so that
\begin{align}
|\mathcal W| & = \sum_{u\in \mathcal I_1} |\mathcal W^{(u)}|
\ll \sum_{u\in \mathcal I_1}U^2T^{2\frac{\log U+\log\log T}{\log L}}  \ \ll \ |\mathcal I_1| U^2 T^{2\alpha+\frac{2\log\log T}{\log L}} \ \ll \ T^{2/5+2/5A+o(1)},
\end{align}
since $|\mathcal I_1| < V\log Q_1 \ll T^{o(1)}$, $U^2 \le Q_1^{2\alpha} \ll T^{o(1)}$, $\log L \ge \log P_1 \ge 5A\log\log T$, by recalling $[P_1,Q_1] = [(\log X)^{33A},(\log X)^{\psi(X)-4A}]$ and $T\in [X^{1/4},X^2]$. 

Hence $A>5$ gives $|\mathcal W| \ll T^{1/2}/Q_2$, and completes the proof of Proposition \ref{prop:MRmain}.



\section{Average Chowla-type correlations}


In this section, we establish the results for higher correlations stated in the introduction.

We first exhibit quantitative cancellation among a broad class of correlations with a `typical' factor $\1_{\mathcal S}\mu$. We use a standard `van der Corput' argument and then apply the key Fourier estimate.
\begin{lemma}\label{lem:Chowlatypical}
Given any $A>5$, $\delta>0$, let $(\log X)^{40A}<H<X$ and $\mathcal S \ = \ \mathcal S(X,H,A,\delta)$ as in \eqref{eq:S}. Suppose $G:\N\to\C$ satisfies $\sum_{n\le X}|G(n)|^2 \ll X(\log X)^{A/20}$. Then
\begin{align*}
\sum_{h\le H}\Big|\sum_{n\le X} \1_{\mathcal S}\mu(n+h)G(n)\Big|
\ \ll_{A,\delta} \ \frac{HX}{(\log X)^{A/40}}.
\end{align*}
\end{lemma}
\begin{proof}
Let $g=\1_{\mathcal S}\mu$. By Cauchy--Schwarz it suffices to show
\begin{align*}
 \frac{HX^2}{(\log X)^{A/20}} \ & \gg \ \sum_{h\le H}  \bigg|\sum_{n\le X} g(n+h)G(n)\bigg|^2 \ = \  \sum_{n,n'\le X}G(n)\overline{G}(n') \sum_{h\le H} g(n+h)g(n'+h).
\end{align*}
Using Cauchy--Schwarz again, the right hand side above is bounded by
\begin{align*}
\sum_{n\le X}|G(n)|^2 \cdot\bigg(\sum_{n,n'} 
\Big|\sum_{h\le H} g(n+h)g(n'+h)\Big|^2\bigg)^{\frac{1}{2}}.
\end{align*}
recalling $g$ is supported on $[1,X]$. By assumption $\sum_{n\le X}|G(n)|^2 \ll X(\log X)^{A/20}$, so it suffices to prove
\begin{align*}
 \frac{H^2 X^2}{(\log X)^{A/5}} \ & \gg \ \sum_{n,n'}\Big|\sum_{h\le H} g(n+h)g(n'+h)\Big|^2
= \sum_{|h|< H}(\lfloor H\rfloor -|h|)\bigg|\sum_n g(n) g(n+h)\bigg|^2.
\end{align*}
But this indeed holds: since $g=\1_{\mathcal S}\mu$, we apply Lemma \ref{lem:Fourier} (Fourier bound) with $f=g=\1_{\mathcal S}\mu$. Thus the trivial bound $F(X)\ll X$ and Theorem \ref{thm:mainFourier} (Key Fourier estimate) give
\begin{align*}
\sum_{|h|\le H}\bigg| \sum_{n\le X} & \1_{\mathcal S}\mu(n)\,\1_{\mathcal S}\mu(n+h)\bigg|^2\\
\ll \ & F(X+2H)\cdot\sup_\alpha \int_0^{X} \bigg|\sum_{x\le n\le x+2H} \1_{\mathcal S}\mu(n)e(n\alpha)\bigg|\dd{x} \ \ll \ \frac{HX^2}{(\log X)^{A/5}}.
\end{align*}
\end{proof}

We now prove the main technical result of the article, which handles averaged Chowla-type correlations with $m\ge1$ copies of the M\"obius function $\mu$ and with any function $G:\N\to\C$ of `moderate growth' which is `amenable to sieves.'

\begin{theorem}\label{thm:mainChowla}
Given any $A>5$, $\delta>0$, let $(\log X)^{40A}<H<X$ and $\mathcal S \ = \ \mathcal S(X,H,A,\delta)$ as in \eqref{eq:S}. Suppose $G:\N\to\C$ satisfies $\sum_{n\le X}|G(n)|^2 \ll X(\log X)^{A/20}$. Then
\begin{align}\label{eq:Chowlafixed}
\sum_{h_1,..,h_m\le H} & \bigg|\sum_{n\le X} G(n)\prod_{j=1}^m \mu(n+h_j)\bigg| \\
& \ \ll_{A,\delta} \ \sum_{h_1,..,h_m\le H} \sum_{n\le X} |G(n)|\prod_{j=1}^m \1_{\overline{\mathcal S}}(n+h_j) \ + \ \frac{mXH^m}{(\log X)^{A/40}}. \nonumber
\end{align}
\end{theorem}
\begin{proof}
We observe from Lemma \ref{lem:Chowlatypical} that any correlation with a factor $\1_{\mathcal S}\mu$ exhibits strong cancellation. So we split up $\mu = \1_{\overline{\mathcal S}}\mu +\1_{\mathcal S}\mu$ until each term has a factor $\1_{\mathcal S}\mu$, except for one term with $m$ factors of $\1_{\overline{\mathcal S}}\mu$. Thus the product in \eqref{eq:Chowlafixed} becomes
\begin{align*}
\prod_{j=1}^m \mu(n+h_j)  = \prod_{j=1}^m\1_{\overline{\mathcal S}}\mu(n+h_j) \ + \ \sum_{i=1}^m \1_{\mathcal S}\mu(n+h_i)\prod_{1\le j< i} \1_{\overline{\mathcal S}}\mu(n+h_j)\prod_{i<j\le m} \mu(n+h_j).
\end{align*}

Hence we bound the left hand side of \eqref{eq:Chowlafixed} by $\Sigma_1 + \Sigma_2$, where
\begin{align}
\Sigma_1 & =  \ \sum_{h_1,..,h_m\le H}|G(n)|\sum_{n\le X} \prod_{j=1}^m \1_{\overline{\mathcal S}}(n+h_j),  \label{eq:nottypicalconv} \\
\Sigma_2 & = \sum_{i=1}^m \ \sum_{h_1,..,h_m\le H}\bigg|\sum_{n\le X}\1_{\mathcal S}\mu(n+h_i) G_i(n)\bigg|, \label{eq:typicalconv}
\end{align}
where $G_i(n) = G(n)\prod_{1\le j< i} \1_{\overline{\mathcal S}}\mu(n+h_j)\prod_{i<j\le m} \mu(n+h_j)$. In particular $|G_i(n)| \le |G(n)|$. 

Thus Lemma \ref{lem:Chowlatypical} applies to each $\1_{\mathcal S}(n+h)G_i(n)$, so that $\Sigma_2$ is bounded by 
\begin{align*}
\Sigma_2 &  \ \ll_A \ \frac{mH^m X}{W^{1/40}}.
\end{align*}
\end{proof}

\subsection{Deduction of results}
For convenience, denote $\psi_\delta(X)=\min\{\psi(X),(\log X)^{1/3-\delta}\}$.
\begin{proof}[Proof of Theorem \ref{thm:Chowktuple}]
Let $G(n) = \prod_{j=1}^k\Lambda(n+a_j)$ for the tuple $\mathcal A=\{a_1,..,a_k\}$. Then $\sum_{n\le X}|G(n)|^2\ll X(\log X)^k$, and by a standard sieve upper bound
\begin{align*}
\sum_{n\le X}G(n)\prod_{j=1}^m \1_{\overline{\mathcal S}}(n+h_j) \ \ll_{m,\mathcal A} \ X\Big(\prod_{p\in[P_1,Q_1]}+\prod_{p\in[P_2,Q_2]}\Big)\Big(1-\frac{1}{p}\Big)^m \ll_{m,\delta,\mathcal A} \frac{X}{\psi_\delta(X)^m},
\end{align*}
using Mertens' product theorem. Hence Theorem \ref{thm:mainChowla} with $A=20(m+k)$ gives
\begin{align}\label{eq:ChowlaHL}
\sum_{h_1,..,h_m\le H} & \bigg|\sum_{n\le X} \prod_{j=1}^k\Lambda(n+a_j)\prod_{j=1}^m \mu(n+h_j)\bigg| 
 \ \ll_{m,\delta,\mathcal A} \ \frac{XH^m}{\psi_\delta(X)^m}.
\end{align}
\end{proof}

\begin{proof}[Proof of Theorem \ref{cor:divisorcorr}]
Let $G(n) = \prod_{i=1}^jd_{k_i}(n+a_i)$ for the tuple $\mathcal A=\{a_1,..,a_j\}$ and recall $k=\sum_{i=1}^jk_i$. Using work of Henriot \cite[Theorem 3]{Hen}, we may obtain
\begin{align*}
\sum_{h\le H}\sum_{n\le X}\1_{\overline{\mathcal S}}(n+h) \prod_{i=1}^j d_{k_i}(n+a_i)
\ \ll_{\mathcal A} \ \frac{HX}{(\log X)^{j+1}}
\sum_{n\le \sqrt{X}}\frac{\1_{\overline{\mathcal S}}(n)}{n}\prod_{i=1}^j\sum_{n\le \sqrt{X}}\frac{d_{k_i}(n)}{n}.
\end{align*}
By the divisor bound $\sum_{n\le \sqrt{X}}d_{k_i}(n)/n\ll X(\log X)^{k_i}$, and by Mertens' product theorem
\begin{align*}
\sum_{n\le \sqrt{X}}\frac{\1_{\overline{\mathcal S}}(n)}{n} \ \ll \ \log X\Big(\prod_{p\in[P_1,Q_1]}+\prod_{p\in[P_2,Q_2]}\Big)\Big(1-\frac{1}{p}\Big) \ \ll_{\delta} \ \frac{\log X}{\psi_\delta(X)}.
\end{align*}
Thus since $\sum_{n\le X}|G(n)|^2\ll X(\log X)^k$, Theorem \ref{thm:mainChowla} with $A=20k$ gives
\begin{align*}
\sum_{h\le H} & \bigg|\sum_{n\le X} \mu(n+h)\prod_{i=1}^j d_{k_i}(n+a_i)\bigg| 
 \ \ll_{\delta,\mathcal A} \ \frac{HX}{\psi_\delta(X)}(\log X)^{k-j}.
\end{align*}
\end{proof}


\subsection{Almost all shifts}


Corollary \ref{cor:exceptmu} follows from the following result by the triangle inequality for $g_j=\mu$.
\begin{theorem}
Suppose $H<X$ and $\log H/\log_2 X \to\infty$ as $X\to\infty$. Let $g_1=\mu$ and take any $g_j:\N\to\C$ with $|g_j|\le 1$ for $1< j\le k$. Then for any fixed shifts $h_2,...,h_k\le H$, $K>0$ we have
\begin{align}
\sum_{p\le X}\prod_{j=1}^k g_j(p+h_j) \ = \ o(\pi(X)),
\end{align}
for all except $O_K(H(\log X)^{-K})$ shifts $h_1\le H$.
\end{theorem}
\begin{proof}
Given $\eps>0$ and fixed shifts $h_2,..,h_k\le H$, we aim to show $|\mathcal E|\ll_{\eps} H(\log X)^{-K}$ for the exceptional set
\begin{align}
\mathcal E = \Big\{h\le H: \Big|\sum_{p\le X}\mu(p+h)\prod_{j=2}^kg_j(p+h_j)\Big|>2\eps \pi(X) \Big\}.
\end{align}

To this, by Markov's inequality we have
\begin{align*}
|\mathcal E|(\eps \pi(X)) 
& \ll \sum_{h\in\mathcal E}\bigg|\sum_{p\le X}\mu(p+h)\prod_{j=2}^kg_j(p+h_j)\bigg| \\
&\le \ \sum_{h\in \mathcal E}\bigg|\sum_{p\le X}\1_{\overline{\mathcal S}}(p+h)\bigg| \ + \
\sum_{h\le H}\bigg|\sum_{p\le X}\1_{\mathcal S}\mu(p+h)\prod_{j=2}^kg_j(p+h_j)\bigg| \\
& \ \ll_A \quad \frac{\pi(X)}{\psi(X)}\sum_{h\in \mathcal E}\prod_{\substack{p\mid h\\p>P_1}}(1+\tfrac{1}{p}) \ \ + \quad  \frac{H\pi(X)}{(\log X)^{A/40}},
\end{align*}
using Lemma \ref{lem:Chowlatypical} when $p+h\in \mathcal S$, and a standard sieve upper bound \cite[Theorem 7.4]{Opera} when $p+h\notin \mathcal S$.
Here $\mathcal S = \mathcal S(X,H,A,\delta)$ as in \eqref{eq:S} with $A=80K$ and $\delta=1/10$, say.

Observe for any $h\le H=(\log X)^{\psi(X)}$ the above product is at most $\prod_{P_1<p\le z}(1+\tfrac{1}{p}) \ll \frac{\log z}{\log P_1}$ where $z = P_1+\psi(X)\log_2 X$. Recalling $P_1 = (\log X)^{33A}$ this gives
\begin{align*}
\frac{\pi(X)}{\psi(X)}\sum_{h\in \mathcal E}\prod_{\substack{p\mid h\\p>P_1}}(1+\tfrac{1}{p}) \ = \ o\big(|\mathcal E|\,\pi(X)\big).
\end{align*}
Hence we conclude $|\mathcal E| \ll \frac{1}{\eps}H(\log X)^{-K}$ as desired.
\end{proof}



\section{Non-pretentious multiplicative functions}

In this section we prove Theorem \ref{thm:pretend}, which we restate below.

\noindent
{\bf Theorem \ref{thm:pretend}}
{\it Let $H = X^\theta$ for $\theta\in(0,1)$, and take a multiplicative function $f:\N\to\C$ with $|f|\le1$. There exists $\rho\in (0,\frac{1}{8})$ such that, if $M(f;X^2/H^{2-\rho},Q) \to \infty$ as $X\to \infty$ for each fixed $Q$, then}
\begin{align*}
\sum_{h\le H}\Big|\sum_{p\le X}f(p+h)\Big| \ = \ o_{\theta,\rho}\big(H\pi(X)\big).
\end{align*}
\begin{proof}
Consider the exponential sum $F_x(\alpha)=\sum_{x\le m\le x+2H}f(m)e(m\alpha)$. The hypotheses of our theorem are made in order to satisfy \cite[Theorem 1.4]{MRTUnif}, which in this case gives
\begin{align}\label{eq:FourUnif}
\int_0^X \sup_\alpha |F_x(\alpha)| \dd x \ = \ o_{\theta,\rho}(HX).
\end{align}
We critically note the supremum is {\it inside} the integral.

Now on to the proof, it suffices to show $S_f = o(HX)$ where
\begin{align*}
S_f := \sum_{h\le H}\Big|\sum_{n\le X}\Lambda(n)f(n+h)\Big| 
\ll \frac{1}{H}\sum_{h\le 2H}(2H-h)\Big|\sum_{n\le X}\Lambda(n)f(n+h)\Big|.
\end{align*}
For each sum $z_h=\sum_{n\le X}\Lambda(n)f(n+h)$, denote 1-bounded coefficients $c(h)$ given by the relation $|z_h|=c(h)z_h$ so that
\begin{align*}
S_f & \ll \frac{1}{H}\sum_{h\le 2H}(2H-h)c(h)\sum_{n\le X}\Lambda(n)f(n+h)\\
& = \frac{1}{H}\sum_{h\le 2H}c(h)\sum_{n\le X}\Lambda(n) \sum_{m\le X+2H} f(m)\1_{m=n+h}\cdot\int_0^X \1_{x\le n,m\le x+2H}\dd{x}\\
& = \frac{1}{H}\int_0^X\int_0^1\sum_{h\le 2H}c(h)e(h\alpha) \sum_{x\le n,m\le x+2H} \Lambda(n)f(m)e\big((n-m)\alpha\big)\dd{\alpha}\dd{x},
\end{align*}
by orthogonality $\int_0^1 e(n\alpha)\dd{\alpha} =\1_{n=0}$. That is, we have the following triple convolution
\begin{align}\label{eq:SfFourier}
S_f \ \ll \ \frac{1}{H}\int_0^X\int_0^1 C_0(\alpha) L_x(-\alpha) F_x(\alpha)\dd{\alpha}\dd{x},
\end{align}
denoting the sums $C_0(\alpha) = \sum_{h\le 2H}c(h)e(h\alpha)$ and $L_x(\alpha)=\sum_{x\le n\le x+2H}\Lambda(n)e(n\alpha)$.

We shall split the inner integral on $\alpha$ according to the size of $L_x$. Specifically, for each $x$ let $\mathcal T_x = \{\alpha\in[0,1] : |L_x(\alpha)| \ge \delta H \}$.
Then by Markov's inequality, $\mathcal T_x$ has measure
\begin{align}\label{eq:measT}
\int_{\mathcal T_x}\dd{\alpha} \le \frac{1}{(\delta H)^4}\int_{\mathcal T_x} |L_x(\alpha)|^4\dd{\alpha} \ll \frac{1}{\delta^4 H},
\end{align}
since the Fourier identity implies
\begin{align*}
\int_0^1 | L_x (\alpha)|^4  \dd\alpha \ & = \ \sum_{x\le n_1,n_2,n_3,n_4\le x+2H} \Lambda(n_1)\Lambda(n_2)\Lambda(n_3)\Lambda(n_4) \1_{n_1+n_2=n_3+n_4} \\
& \ll \  (\log X)^4\sum_{\substack{x\le p_1,p_2,p_3,p_4\le x+2H\\p_1+p_2=p_3+p_4}}1  \ \ll_{\theta} \ H^3,
\end{align*}
by a standard sieve upper bound \cite[Theorem 7.4]{Opera}. Thus plugging \eqref{eq:measT} into \eqref{eq:SfFourier}, we obtain
\begin{align*}
S_f \ll \frac{1}{H}\int_0^X\int_{[0,1]\setminus \mathcal T_x} C_0(\alpha) L_x(-\alpha) F_x(\alpha)\dd{\alpha}\dd{x} \ + \ \frac{1}{\delta^4 H^2}\int_0^X \sup_{\alpha\in \mathcal T_x}|C_0(\alpha) L_x(-\alpha) F_x(\alpha)|\dd{x}.
\end{align*}

Denote the two integrals above by $I_1$ and $I_2$. Observe $I_2 \,\ll_{\theta}\, \delta^{-4}\int_0^X \sup_{\alpha}|F_x(\alpha)|\dd{x}$,
using $|C_0(\alpha)| \le H$ trivially and $|L_x(\alpha)| \ll_{\theta} H$ by the Brun--Titchmarsh theorem. Then by definition of $\mathcal T_x$, Cauchy--Schwarz implies
\begin{align*}
I_1 \ &\le \ \delta\int_0^X\int_{[0,1]\setminus \mathcal T_x} |C_0(\alpha) F_x(\alpha)|\dd{\alpha}\dd{x}\\
& \le \ \delta\int_0^X
\bigg(\int_0^1 |C_0(\alpha)|^2\dd{\alpha} \cdot \int_0^1 |F_x(\alpha)|^2\dd{\alpha}\bigg)^{1/2}\dd{x} \ \ll \ \delta HX,
\end{align*}
by Parseval's identity applied to $C_0$ and $F_x$. Thus combining bounds for $I_1,I_2$ gives
\begin{align}
S_f \ll_{\theta} \delta HX \ + \ \delta^{-4}\int_0^X \sup_{\alpha}|F_x(\alpha)|\dd{x}.
\end{align}
Hence taking $\delta\to0$, the Fourier uniformity bound \eqref{eq:FourUnif} gives $S_f = o_{\theta,\rho}(HX)$ as claimed.
\end{proof}

\section*{Funding}
This work was supported by a Clarendon Scholarship at the University of Oxford.

\section*{Acknowledgments}
The author is grateful to Joni Ter\"av\"ainen for suggesting the problem and for many valuable discussions. The author thanks Joni Ter\"av\"ainen, James Maynard, and the anonymous referee for careful readings of the manuscript and for helpful feedback.

\bibliographystyle{amsplain}

\end{document}